\documentclass[english]{amsart}
\usepackage[T1]{fontenc}
\usepackage{color}
\usepackage{textcomp}
\usepackage{subfig}
\usepackage{amsmath}
\usepackage{amssymb}
\usepackage{amsthm}
\usepackage[english]{babel}
\usepackage{graphicx}
\theoremstyle{definition}

\newtheorem{theorem}{Theorem}
\newtheorem{lemma}[theorem]{Lemma}
\newtheorem{proposition}[theorem]{Proposition}

\newenvironment{proofof}[1]{\noindent {\bf{Proof of #1.}}}{ \hfill\qed\\ }

\newcommand{\eqdef}{\mathrel{\mathop=}:}
\def\N{\mathbb{N}}
\def\Z{\mathbb{Z}}
\def\A{\mathcal{A}}
\def\a{\vec{a}}

\def\x{\vec{x}}
\def\y{\vec{y}}
\def\z{\vec{z}}

\def\A{\mathcal{A}}
\def\L{\mathcal{L}}
\def\BL{\mathcal{BL}}

\def\X{X_{aux}}

\begin{document}
\title[Entropy and complexity of spy billiards]{Entropy and complexity of polygonal billiards\break  with spy mirrors}
\author {Alexandra Skripchenko}
\address{Faculty of Mathematics, National Research University Higher School of Economics, Vavilova St. 7, 112312 Moscow, Russia}
\email{sashaskrip@gmail.com}

\def\curraddrname{{\itshape Address}}
\author{Serge Troubetzkoy}
\address{Aix Marseille Universit\'e, CNRS, Centrale Marseille, I2M, UMR
  7373, 13453 Marseille, France}
\curraddr{ I2M, Luminy\\ Case 907\\ F-13288 Marseille CEDEX 9\\ France}
 \email{serge.troubetzkoy@univ-amu.fr}

\begin{abstract}
We prove that a polygonal billiard with one-sided mirrors has zero topological entropy.  In certain cases we show
sub exponential and for other polynomial estimates on the complexity.
 \end{abstract}
 \maketitle
\section{Introduction}
\subsection {Polygonal billiards with one-sided mirrors}
We consider a table  consisting of a polygon $Q\subset \mathbb R^{2}$ (not necessarily rational) with several one-sided mirrors inside; i.e.\ a straight line segments connection pairs of  points in $Q$, each of
which  has two sides, a transparent side and
a reflecting side. The billiard  is defined as follows. Consider a point particle and a direction $\theta\in \mathbb S^{1}$; the point moves in the direction $\theta$ with a unit speed up to the moment when it reaches the boundary, if it arrives at a transparent side of a mirror it passes through it unperturbed, while if it arrives at  a reflecting side of a mirror or at the boundary of  $Q$ it is reflected with the usual law of geometric optics, the angle of incidence equals the angle of reflection.

Polygonal billiards with one-sided mirrors were described for the first time by M.\ Boshernitzan and I.\ Kornfeld in \cite{BK}, in this article one-sided mirrors were called spy mirrors. However, they considered the
less general case of rational polygons with the mirrors that form rational angles with the sides of polygonal table.  Such tables give rise to interval translation maps, a generalization of
interval exchange maps.  In contrast of interval exchange transformations, interval translation maps are poorly understood, only a few
results are known (see \cite{BK}, \cite{SchT}, \cite{Ba}, \cite{BT}, \cite{SIA},  \cite{BC}, \cite{V}, \cite{ST}).
A particular example of a rational billiard with one-sided mirrors,  the square with a vertical one-sided mirror with one end point on the bottom side of the square, was studied in \cite{ST}.

In this article we will prove two types of results. In the setting of an arbitrary polygon with one-sided mirrors we show that the topological entropy of our system is zero.  In certain more
restricted settings we show that we have sub exponential or polynomial growth estimates. The next two subsections describe these results in more 
detail.

\subsection{Topological entropy}
We prove that the polygonal billiard with one-sided mirrors has zero topological entropy.  To
do this we first consider 
the inverse limit space of a polygonal billiard with one-sided mirrors and show that it has zero topological entropy (an exact statement is provided below). We show that the inverse limit space in our case is closely related with the attractor (the notion of the attractor of the billiard map was introduced in \cite{ST}) and that the attractor has zero topological entropy. Then we extend the zero entropy result to the
full phase space.

There exist several different proofs of zero topological entropy for polygonal billiards without  one-sided mirrors (see \cite{K}, \cite{GKT}, \cite{GH}). Our proof mainly uses some ideas from \cite{GKT} and \cite{K}. The main difference with the situation of classical billiards is the  non-invertibility of our system. 
Also, J. Buzzi in \cite{Bu} showed a closely related result that piecewise isometries in any dimension are of zero entropy. 
For rational polygonal billiards with rational one-sided spy mirrors, zero topological entropy is a corollary from the fact that the directional complexity is at most polynomial \cite{Ba}, and the variational principle.

Throughout the article the term {\it side} will denote a side of the 
polygon $Q$ or a side of a one-sided mirror, and the term {\it vertex}  denotes an end point of a side. 
 We will are denote by $q$ the number of sides of $Q$ and $r$ the number of spy mirrors, thus
we have $q + 2r$ sides. The collection of sides
will be call the {\em boundary} $\Gamma$.
We will consider the {\em billiard  map} $T$, the first return map to  $\Gamma$. 
The {\em phase space} $T\Gamma$ of the billiard map is the subset of inner
pointing vectors of unit tangent bundle (for vectors with base point in a one-sided mirror this means
that if we reverse the direction of the vector it will point at the reflecting side of the mirror).
Note that if the billiard orbit arrives at a vertex of $Q$ then the collision rule is not well defined
since we can reflect with respect two different sides, thus the billiard map is not defined for such points.
 Let $\pi: T\Gamma \to \Gamma$ denote the natural projection.

Let $I_i : 1 \le i \le q + 2r$ be an enumeration of the sides of $Q$.
The forward orbit of a point $x$ can be coded by the sequence  of sides  hit by the orbit.
Let
$$\Sigma^+ \eqdef \{ \a := (a_i)_{I \in \N}: \exists x \text{ such that } \pi(T^ix) \in I_{a_i} \ \forall i \ge 0\}.$$
In this definition it is implicitly assumed that the map $T^ix$ is defined for all $i \ge 0$. 
We use the discrete topology on the collection of sides, and the product topology on $\overline{\Sigma^+}$.
The left shift map on $\overline{\Sigma^+}$ will be denoted by $\sigma$. The main result of this section is the following theorem.

\begin{theorem}\label{zero''} For any polygon with spy mirrors we have
$$h_{top}(\overline{\Sigma^+},\sigma) = 0.$$
\end{theorem}
Suppose that $\mu^+$ is an invariant measure on $\overline{\Sigma^+}$, Theorem \ref{zero''} implies that $\sigma$ is
$\mu^+$ almost surely invertible.
The {\em  complexity} $p(n)$ is the number of words on length $n$ which appear in $\Sigma^+$.  Theorem \ref{zero''}
implies that $\lim_{n \to \infty} \log(p(n))/n = 0$. 

Our proof of Theorem 1  uses invertibility,  we begin by working in the  inverse limit of the coding space  
$$\Sigma \eqdef  \{ \a := (a_{i})_{I \in \Z}: (a_{i-j})_{i \ge j} \in \Sigma^+ \ \forall j \in \Z\}.$$
We also introduce
$$\Sigma^- \eqdef \{ \a := (a_i)_{I \in \N}: \exists (b_i)_{i \in \Z} \in \Sigma \text{ such that } a_i = b_i \ \forall i \le 0\}.$$
Finally the inverse limit of the billiard map is
$$\Omega \eqdef \{ \x := (x_i)_{i \in \Z}: Tx_i = x_{i+1} \ \forall i \in \Z \}.$$
Here we again assume that the forward orbit $T^jx_i$ is definite for all  $i$ and all $j \ge 0$.
We use the natural topology of $T\Gamma$ on $x_0$ and the product topology on $\Omega$.

 The attractor of the billiard map is the set $\A := \cap_{n \ge 0} T^n (T\Gamma)$, where for
 the set $T^n(T\Gamma)$ we consider
 only the points in $T\Gamma$ for which $T^n (T\Gamma)$ is defined.

All shift maps (on $\Sigma$ or $\Omega$) will be denoted by $\sigma$.   We extend the shift map to $\overline{\Omega}$, $\overline{\Sigma^+}$ and $\overline{\Sigma}$. We show that

\begin{theorem}\label{zero} For any polygon with spy mirrors we have
$$h_{top}({\A},T) = h_{top}(\overline{\Omega},\sigma) = h_{top}(\overline{\Sigma},\sigma) = 0.$$
\end{theorem}
\noindent
and then we show that Theorem \ref{zero''} follows from Theorem \ref{zero}.

\subsection{Complexity estimates in special cases} 

A {\em generalized diagonal}  is an orbit segment which starts and ends in a vertex of $Q$,
let $N_{vert}(n)$ denote the number of generalized diagonals of combinatorial length at most $n$.

We begin with a general theorem, and then we will apply it to specific examples.
\begin{theorem}\label{thm1}
Suppose that $Q$ is a $q$-gon with $r$ spy mirrors, then 
$$p(n)
\le 1 + (q+2r- 1)n  +  \left (2((q+2r)^2-3)\sum_{j=0}^{n-1} \sum_{i=0}^{j} N_{vert}(i) \right ).$$
\end{theorem}

We call $Q$ a  {\em symmetric polygon with spy mirrors}  if there is a polygon $P$ such that $Q$ is obtained from $P$ via a finite unfolding
and there are finitely many spy mirrors which are contained in the common edges of the unfolded copies of $P$ (see Figure \ref{s}).

 \begin{figure}[h]
\vspace{-4cm}
\includegraphics[width=7.1cm,height=10cm]{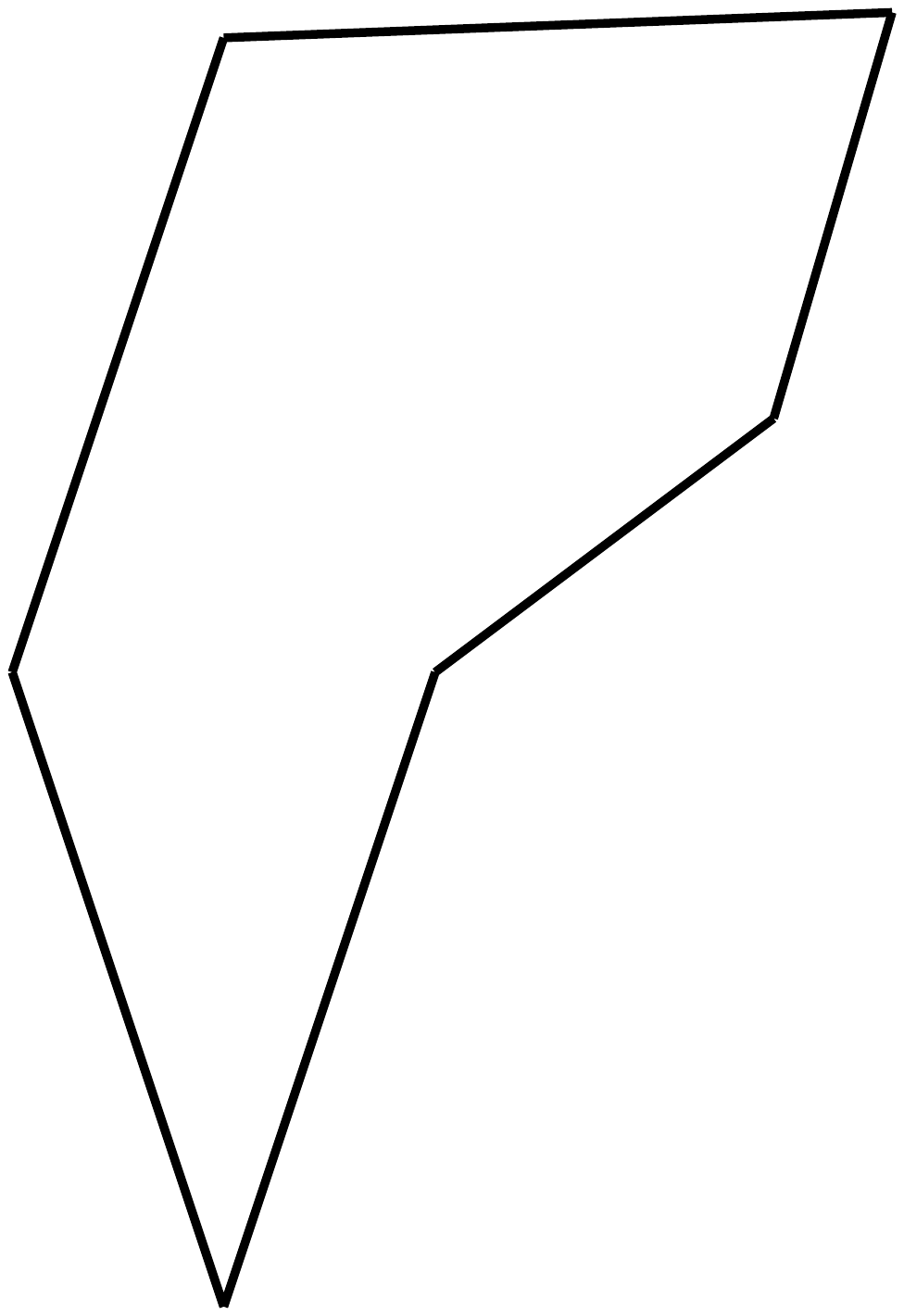}
\vspace{-1.5cm}
\caption{A symmetric polygon with two spy mirrors}
\label{s}
\end{figure} 

\begin{theorem}\label{thm2}
Suppose that $Q$ is  a rational symmetric  polygon with spy mirrors. Then there is a constant $C > 0$ so that
the total  complexity satisfies $p(n) \le Cn^4$ for all $n \ge 0$. 
\end{theorem}

Next we consider symmetric polygon with spy mirrors obtained from a triangle (non necessarily rational). 
Two smallest angles determine a triangle up to scaling, and billiards are scaling invariant.  Thus, up to scaling, the set of triangles is
a subset of $\mathbb{R}^2$ equipped with Lebesgue measure. In the next theorem the word {\em typical} will mean Lebesgue almost every.
\begin{theorem}\label{thm3}
Suppose that $Q$ is a symmetric polygon with spy mirrors obtained from a typical triangle.  Then 
for every $\varepsilon > 0$ there is a constant $K > 0$ so that $p(n) \le Ke^{n^\varepsilon}$ for all $n \ge 0$.

\end{theorem}
Also one can prove a special complexity estimation for a generalization of the billiard with square table that we studied in \cite{ST}.
\begin{theorem}\label{thm4}
Suppose that $Q$ is  the square with $k$  vertical spy mirrors. 
Then there is a constant $K > 0$ so that the total complexity satisfies
$p(n) \le K n^{k+4}$ for all $n \ge 0$. 
\end{theorem}

\section{The proofs of the entropy results}
\subsection{Unfolding and strips} 
Consider the backwards billiard flow starting from a point in $\Gamma$; instead of reflecting the orbit about a side of $Q$  we  reflect the polygon about the same side and continue the orbit as a straight line. When it meets another side of the reflected polygon we repeat the procedure with respect to this side, etc. We can continue up to the moment when we hit the vertex. The copies of $Q$ obtained after such a reflection we will label with respect to a side that was an axis of reflecting ($Q_{A}$, for instance.)

\begin{lemma}
\label{parallel} Suppose $\x,\y \in \Omega$ and that $x_0$ and $y_0$ are not parallel,
then their backward codings can not coincide.
\end{lemma}
\begin{proof}
We look at unfolding lines for the past orbits, see Figure \ref{1}, we suppose there codes coincide
for a certain interval of times. We remark that when
a backward orbit hits a one-sided mirror, there are two possible preimages, by definition these preimages 
have different codings; thus for the interval of times when the backward codings coincide the choice of
preimages (which is given since $\x$ and $\y$ are in the inverse limit space $\Omega$) is the same for $\x$ and for $\y$.

These lines are eventually linearly divergent,  thus the distance between them is eventually more than twice the diameter of $Q$, and so the backwards unfoldings of $x_0$ and $y_0$ must be different, i.e.\ the reflections must occur in different edges, so 
the codings are different. 
\end{proof}

 \begin{figure}[h]
\vspace{-5cm}
\includegraphics[width=9cm,height=11.5cm]{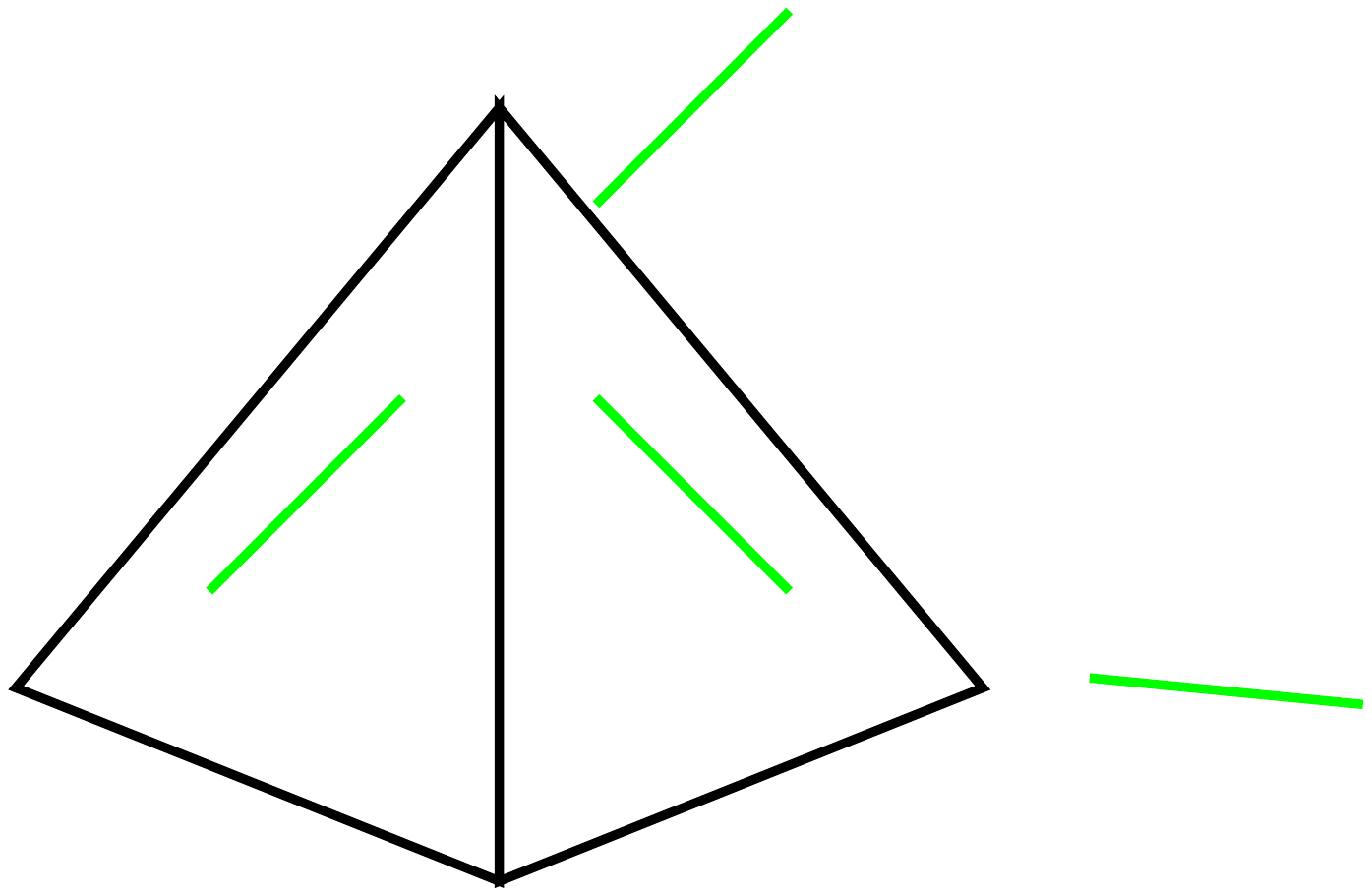}
\vspace{-2cm}
\put(-165,110){$A$}
\put(-207,120){$B$}
\put(-185,110){$M$}
\put(-197,80){$C$}
\put(-117,65){\vector(-1,1){30}}
\put(-120,60){$Q_A$}
\put(-62,120){\vector(-1,0){45}}
\put(-60,115){$Q_{AM}$}
\put(-225,62){\vector(1,1){30}}
\put(-230,55){$Q$}
\put(-75,170){\vector(-1,0){45}}
\put(-72,168){$Q_{AMA}$}
\put(-25,110){\vector(-1,-1){30}}
\put(-23,107){$Q_{AMC}$}
\caption{Non-parallel orbits have different coding}
\label{1}
\end{figure} 

\subsection{Uniqueness of the coding}

For each $\a \in \Sigma$, there exist at least one point  $\x \in \Omega$ such $\pi(\x_i) = \a_i$ for all $i \in \Z$, we denote the set of such $\x$ by $X(\a)$.
The set $X(\a)$ can also be defined for $\a \in \Sigma^-$, and since
for $\x\in \Omega$ from the definition it is immediate that $(x_i)_{i \le 0}$ determines $\x$ uniquely, it follows that 
 $X(\a) \subset \Omega$ for points $\a \in \Sigma^-$.
Lemma \ref{parallel} implies that the set $X(\a)$ consists of parallel points.

Fix $\a \in \Sigma^-$.  
We call $S \subset X(\a)$ a \emph{strip} if 
the set $\{x_0: \x \in S\}$ consists of 
parallel vectors whose base points form an interval.  Note that
the backwards
orbit of all the $\x \in S$ are nonsingular.
Using the unfolding procedure, we will think of $S$ as being a geometric strip in $\mathbb{R}^2$,
hence the name.
We have


\begin{proposition}\label{unique}
Suppose that $Q$ is an arbitrary polygon with a finite number of spy mirrors.
\begin{enumerate}
\item{}  For any $a\in \Sigma^{-}$ which is not periodic the set $X(\a)$ consists of only one point. 
\item{} For any $a \in \Sigma^{-}$ which is periodic the set $X(\a)$ consists of a finite union of 
parallel  strips.
\end{enumerate}
\end{proposition}

\begin{proof}
 
\begin{figure}[h]
\vspace{-4.5cm}
\includegraphics[width=9cm,height=11.5cm]{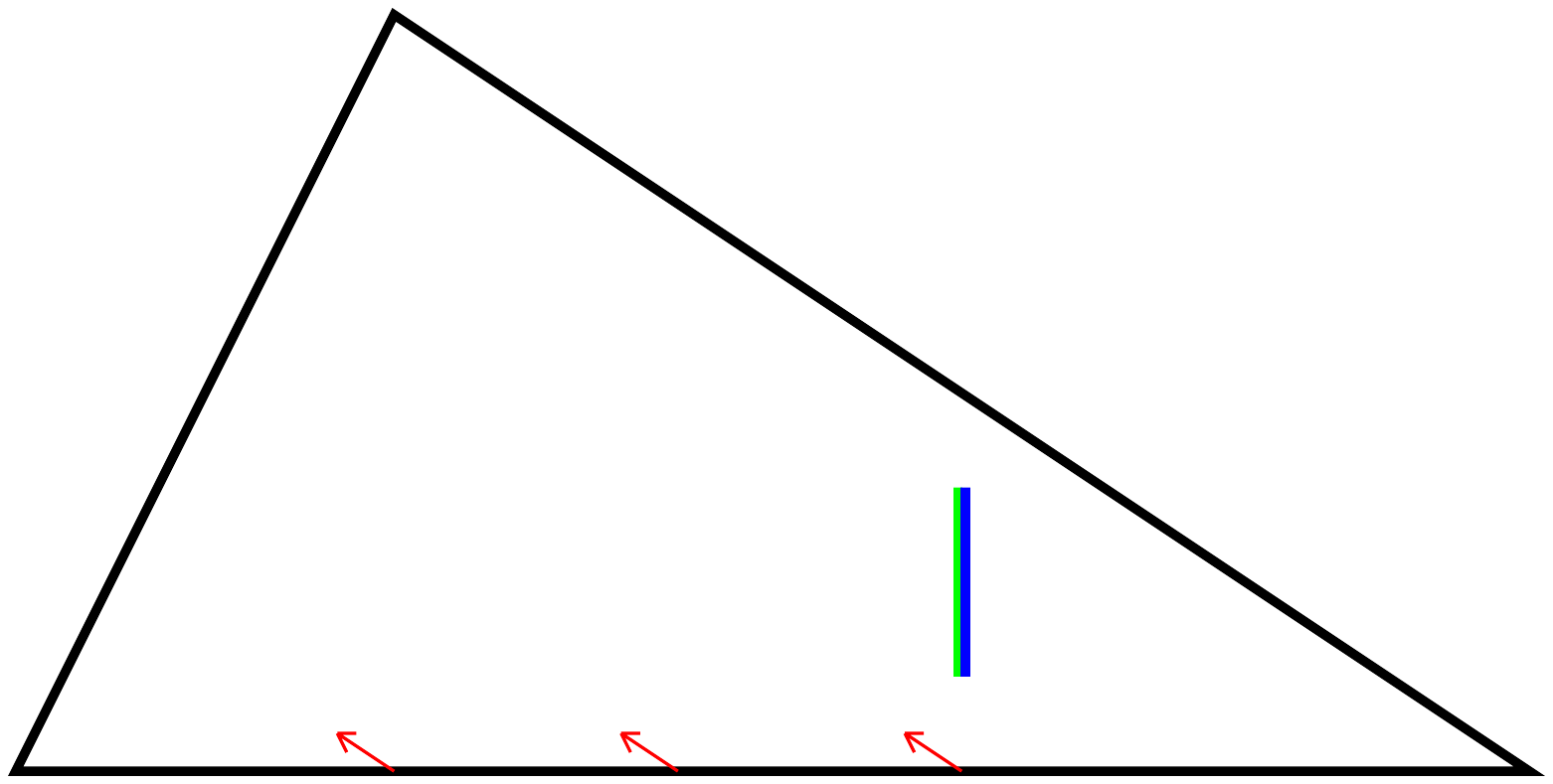}
\put(-180,100){$x$}
\put(-144,100){$z$}
\put(-109,100){$y$}
\vspace{-3.cm}
\caption{The orbit in the middle of the strip has different dynamics.}
\label{2}
\end{figure} 


(1)  The proof is by contradiction, consider two points $x,y \in \Omega$ with exactly the same aperiodic backwards coding $\a$,  Lemma \ref{parallel} implies
that  $x_i$ and $y_i$ are parallel for all $i \le 0$.
It may happen that the trajectories of points between them are not exactly the same (see Figure \ref{2} for a possible example). Note that this can not happen  for the usual billiard in
a simply connect polygon. In order to avoid this problem, and 
enable us to  use the strip argument from \cite{GKT}, we construct an auxiliary dynamical system by declaring that for  all points in between our two special have the same dynamics. Thus, with respect to this
auxiliary dynamics the set $S:= \X(\a)$ is a strip.  More precisely,  the auxiliary dynamics is defined as the map $g$ which rigidly 
maps the vectors in the strip with base-point in the
interval $(x_{i-1},y_{i-1})$ onto vectors in the strip with  the base-point in the interval $(x_{i},y_{i})$ (for all $i \le 0$). 

Consider the sequence of vectors $\z := (z_i)_{i \le 0}$ in the middle of the strip $S$, and the $\alpha$-limit set $Z$
of the auxiliary dynamics $g$ of $\z$. The set $Z$ is compact, and  since strips can not contain vertices  the inverse map $g^{-1}$ is continuous
on $Z$, thus 
we can apply the Birkhoff recurrence theorem,  to conclude that $Z$ contains a uniformly recurrent point $\x^*$ for the map $g^{-1}$. Fix a
sequence $n_i$  so that $g^{-n_i}\z \to \x^*$.
 We consider images of the original strip $S$ under  $g^{-n_{i}}$ and denote it by $S_{i}$. 
 Note that the widths of the $S_i$ are all greater than or equal to the width of $S$, and that $S_i$ converge to a strip $S(\x^*)$ centered at $x^*$ of the same width or more.
We can suppose (by re-enumerating the sides)  that $\x_0^{*}\in I_{1}$.

The point $\x_0^*$  is not tangent to the side $I_{1}$ since then the orbit of $\z$ would come arbitrarily close to one of the endpoints of the side $I_1$ which contradicts the positive width of the strip $S$.

We consider the set of points ${S}^\infty(\x^*)$ having the same
auxiliary backwards code as $\x^*$. Clearly ${S}^{\infty}(\x^*)$ is a strip and  ${S}^\infty(\x^*)  \supset S(\x^*)$.
The left and right boundaries of ${S}^{\infty}(\x^*)$ are two lines, we consider their $\varepsilon$-neighborhoods $N_{\varepsilon}^{L}$ and $N_{\varepsilon}^{R}$.
By the uniform recurrence of $\x^{*}$, vertices fall inside each of $N_{\varepsilon}^{L}$ and $N_{\varepsilon}^{R}$ with bounded gaps between their occurrences
(see Figure \ref{3}).  Fix one of the sides, say $N_{\varepsilon}  := N_{\varepsilon}^{L}$. 
Now $S_i \to S(\x^*)$, this can happen in two ways,  either
$S_i$ is parallel to $S(\x^*)$ for
all sufficiently large $i$ or  not.
In the second case 
the set
$N_{\varepsilon}\cap S_{i}  \subset N_{\varepsilon} \cap S$ contains  an $\varepsilon$-width rectangle with a height $L_{i}$ that goes to infinity as $i\to\infty$; thus  there must be
 a vertex in $S$ (see Figure \ref{3}(b)).  This contradicts  the fact that the backwards orbit
 of all $ \x \in S = X(\a)$ are non-singular,
 thus we can only have a single point $\x \in \Omega$ with aperiodic backwards coding $\a$.
 In the first case for sufficiently large $i$, since they are parallel  we  have $S_i \subset {S}^{\infty}(\x^*)$.
 We consider the set of points $S_i^{\infty}(\z)$ have the same auxiliary backward code as
 $g^{-n_i}\z$, clearly this is the maximal width strip such that $S_i^{\infty} \supset S_i$. By
 maximality $S_i^{\infty} = S^{\infty}(\x^*)$. This implies that $\x^*$ is periodic which contradicts
 the assumption that $\a$ is not periodic. Thus also in this case we can only have a single
 point $\x \in \Omega$ with periodic backwards coding $\a$.

\begin{figure}[h]
\centering
\vspace{-3cm}
\mbox{\subfloat{
\includegraphics[scale=0.3]{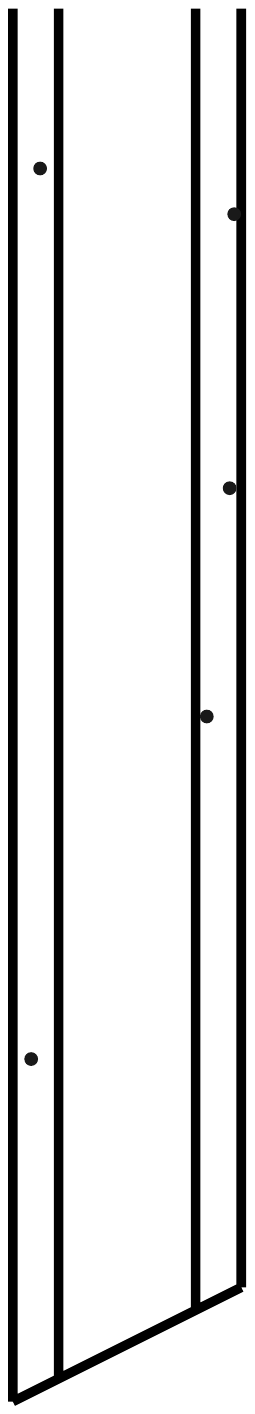}}
\put(-142,170){$N_{\varepsilon}^{L}$}
\put(-130,172){\vector(1,0){39}}
\put(-58,172){\vector(-1,0){19}}
\put(-57,170){$N_{\varepsilon}^{R}$}
\put(75,150){\vector(-1,0){53}}
\put(75,148){$S^{\infty}_{i}$}
\put(-55,140){\vector(-1,0){29}}
\put(-52,137){$S^{\infty}(x^{*})$}
\put(-17,140){\vector(1,0){60}}
\put(0,119){\line(1,0){30}}
\put(0,108){\line(1,0){30}}
\put(0,119){\line(0,-1){11}}
\put(-10,112){$L_{i}$}
\put(-134,143){\vector(2,1){42}}
\put(-134,142){\vector(3,-4){40.5}}
\put(-90,50){$(a)$}
\put(35,50){$(b)$}
\put(-190,140){\footnotesize{vertices or}}
\put(-194,130){\footnotesize{reflecting vertices}}

\quad
\subfloat{
\hspace{-2.5cm}
\includegraphics[scale=0.3]{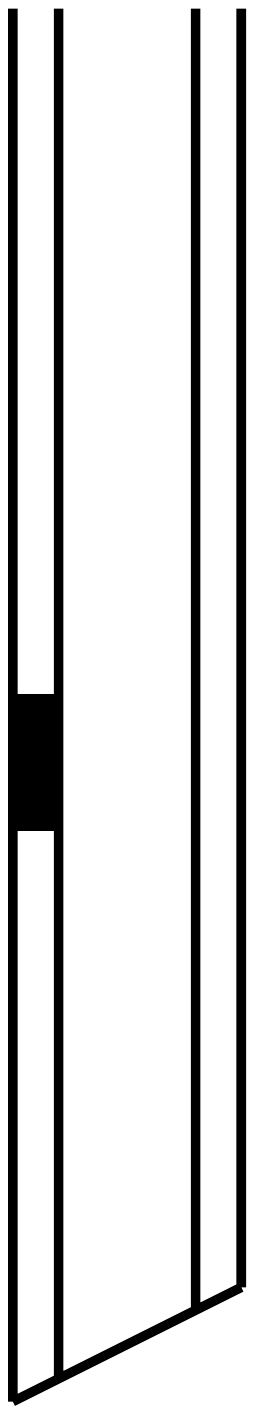}}}
\vspace{-2cm}
\caption{Vertices and recurrence points}
\label{3}
\end{figure}

 (2) Now suppose that $\a \in \Sigma^-$ is periodic.  
 Lemma \ref{parallel} implies that $X(\a)$ consists of parallel
 orbits.  
By the definition of 
the auxiliary dynamics, the set $\X(\a)$ with respect to the auxiliary dynamics is a strip.
Suppose that the period of $\a$ is $k$ and that 
 $\x \in \X(\a)$. From the
 periodicity of $\a$ we conclude that  $x_{i-k} = x_i$,  or $T^k(x_{i-k}) = x_i$, 
 for  all $i \le 0$.  But since $x_i$ for $i > 0$ is defined as $T^i(x_0)$ we conclude that $\x$ is
 $k$ periodic.   Since the map $T^k$ is a local isometry we conclude that  furthermore the width
 of the strip $\X(\a)$ is strictly positive.

\begin{figure}[h]
\vspace{-2.5cm}
\includegraphics[width=7cm,height=7.5cm]{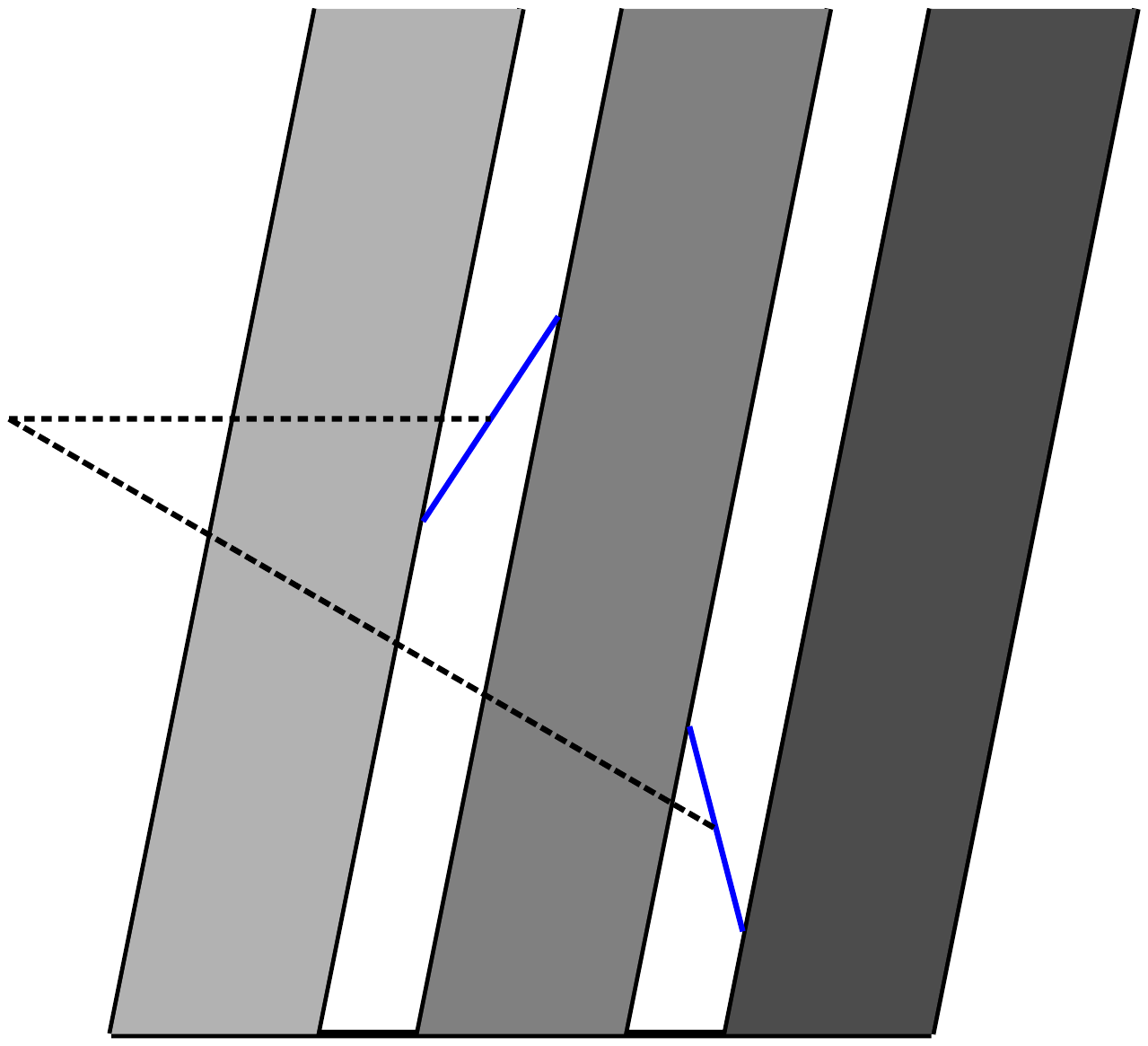}
\put(-200,95){$mirrors$}
\vspace{-1.5cm}
\caption{Finite union of strips}
\label{4}
\end{figure}

Consider the set $\X(\a) \cup T^{-1} \X(\a) \cup \cdots \cup T^{-k+1} \X(\a)$. Since we only consider
 a finite piece of the orbit the reflecting mirrors create at most finitely many "holes" in the auxiliary dynamics strip as in Figure \ref{4},  more precisely points for  which the auxiliary dynamics and the real dynamic 
 disagree, the set $X(\a) \subset \X(\a)$   is a finite union of strips.
 \end{proof}

Let
 $\partial \Sigma :=\overline{\Sigma} \setminus \Sigma$ and
  $\partial \Omega :=\overline{\Omega} \setminus \Omega$.
We extend $X$ to $\overline{\Sigma}$ as follows.
Suppose $\a \in \partial \Sigma$, let $\a^{(n)} \in \Sigma$ be such that $\a^{(n)} \to \a$. Let $\x^{(n)} \in X(\a^{(n)})$. Since $\overline{\Omega}$ is closed there is
an $\x \in \overline{\Omega}$ which is a limit point of the $\x^{(n)}$.  Let $X(\a)$ be the collection of all such $\x$.
We remark that all such points $\x \in \partial \Omega$
since $\a \not \in \Sigma$ .

\begin{proposition}\label{extend} If $\x \in \partial \Omega$ then there exist at most countably many $\a \in \partial \Sigma$ such
that $\x \in X(\a)$.  
\end{proposition}

\begin{proof}
Suppose that $\x \in \partial \Omega$. Consider $\x^{(n)} \in \Omega$ such that $\x^{(n)} \to \x$.  
Let $\a^{(n)}$ be the code of $\x^{(n)}$, and $\a$ any limit point of the $\a^{(n)} $, then
clearly $\a \in \partial \Sigma$, and $\x \in X(\a)$.
At a certain time the orbit of $\x$ reaches a vertex.  There are clearly at most $q + 2r$ possible extensions by continuity
of the dynamics and of the code  (see, for example, Figure \ref{2015}  where the purple lines are the orbits that hit some vertex and several dotted lines represent possible extensions).  Each of these extensions may again reach a vertex,
so again each of them has at most $q+2r$ possible extensions.  This can happen at most
a  countable number of times.
\end{proof}

\begin{figure}[h]
\centering
\vspace{-3.5cm}
\hspace{-1cm} \mbox{\subfloat{
\includegraphics[scale=0.3]{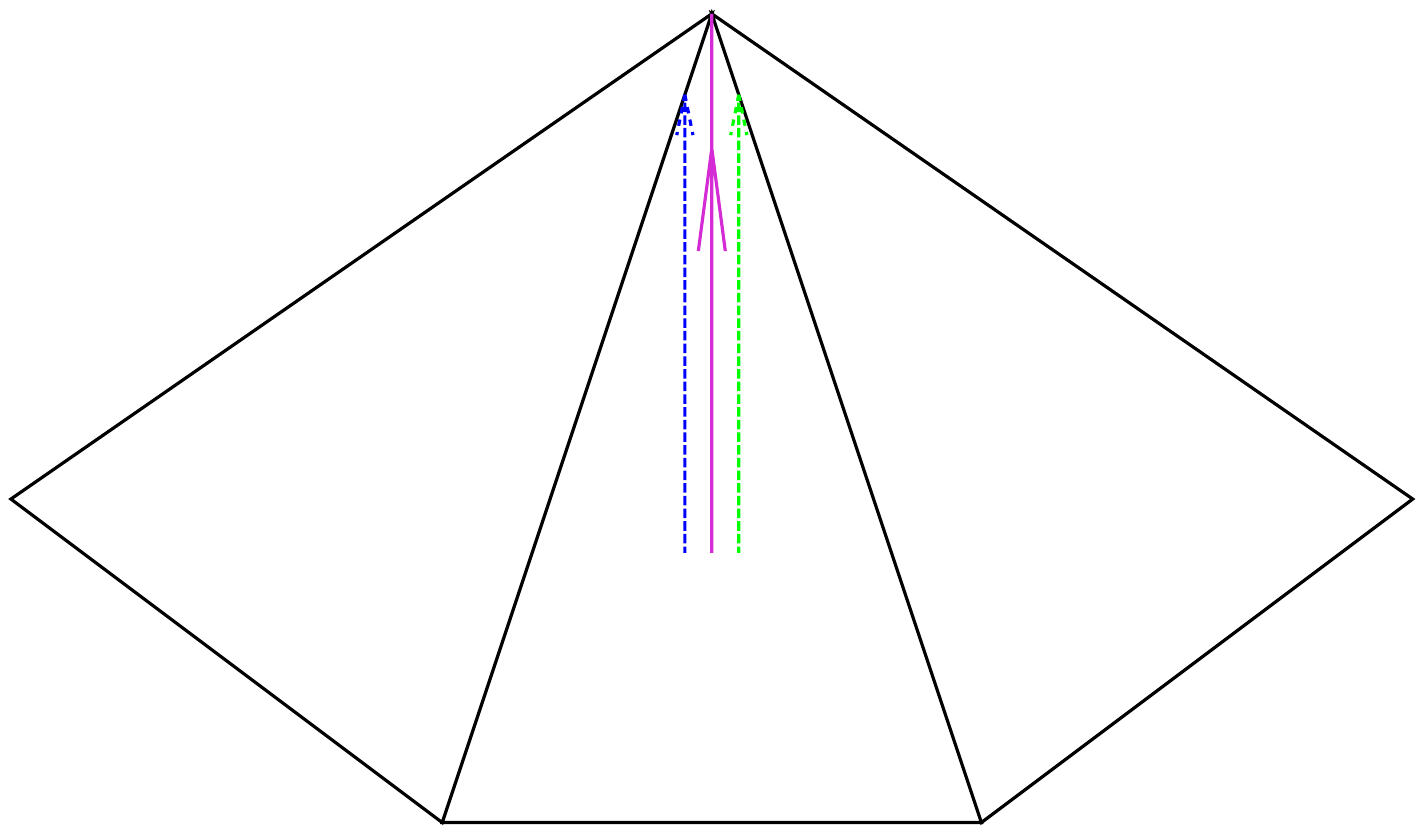}}
\hspace{-1.5cm}\mbox{\subfloat{
\includegraphics[scale=0.3]{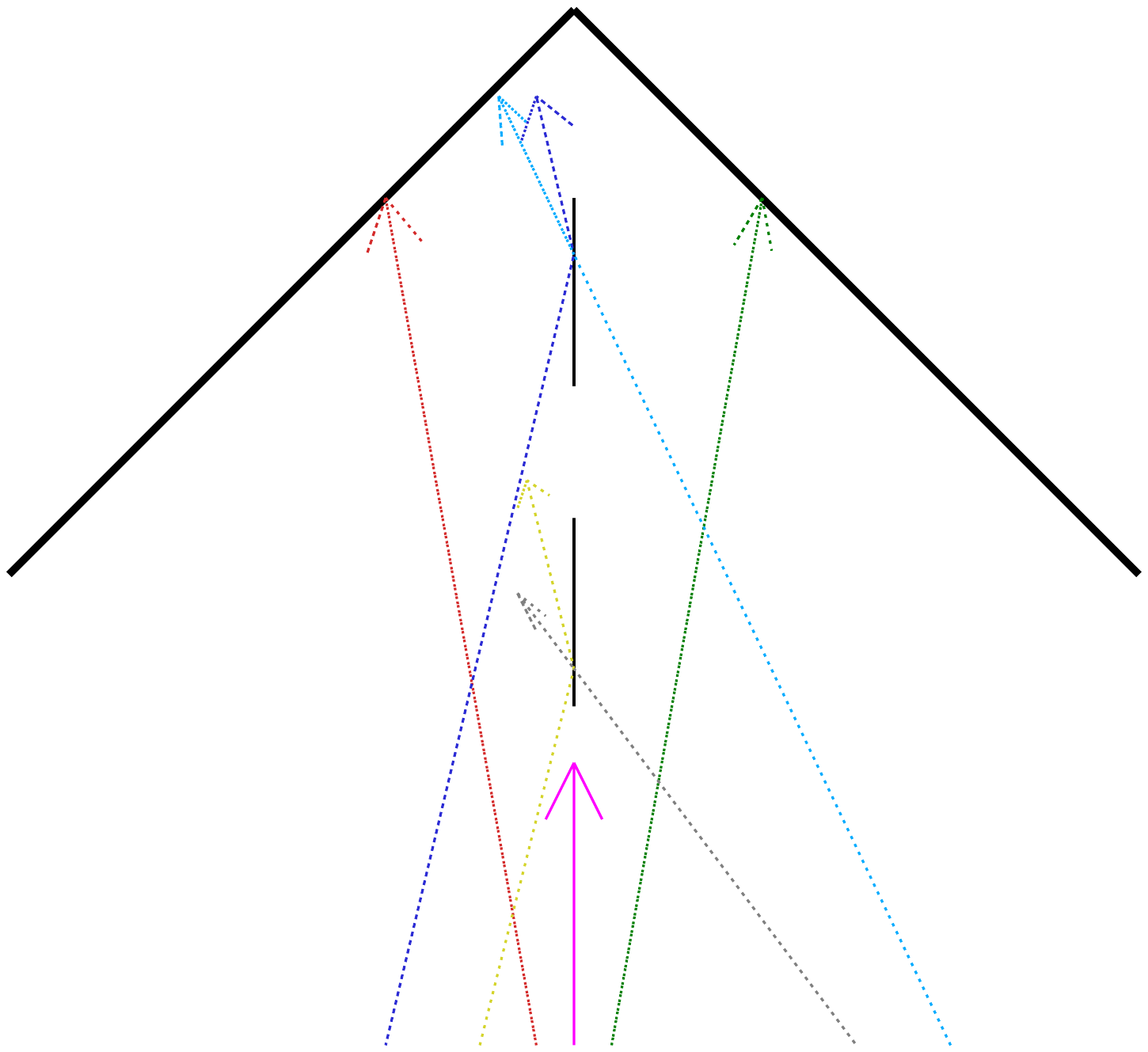}}}}
\vspace{-1.5cm}
\caption{The purple trajectory hits a vertex, all possible symbolic extensions are indicated via
close by orbits, there are 2 extensions for the left figure, 6 for the right figure.}
\label{2015}
\end{figure}

\subsection{The attractor}

For $\x$ in ${\Omega}$ let $Y(\x) = x_0$. 
The following proposition shows that $({\A},T)$ is a continuous factor of $({\Omega},\sigma)$.

\begin{proposition}\label{attractor}
(i) $\x \in \Omega \iff Y(\x) \in \A$,\\
(ii) $Y(\sigma \x) = T Y(\x)$,\\
(iii) the map $Y$ is continuous.
\end{proposition}
\begin{proof}
(i) $\Longrightarrow$ Suppose $\x \in \Omega$, then $T^n(x_n) = x_0$ since $T(x_i) = x_{i+1}$ for all $i$.

\noindent$\Longleftarrow$ Suppose $x_0 \in \A$, then the set $\{T^{-i}(x_0)\}$ is non-empty for all $i \ge 0$.
Consider the tree structure on these sets, i.e.\ draw an arrow from  $x_i \in \{T^{-i}(x_0)\}$ to $x_{i+1} \in \{T^{-i-1}(x_0)\}$
iff $T(x_i) = x_{i+1}$.  Since the vertex set on level $n$ of the tree is non-empty for all $n$, there must be an infinite path in this 
tree, which defines the past of $\x := (x_i)$, and the future is defined by $T^i(x_0)$ with $i > 0$.

(ii)  and (iii) follow immediately from the definition of $\Omega$.
\end{proof}

\subsection{Entropy of ergodic measures} This subsection follows ideas in \cite{K}.
First, we show that nothing interesting happens on the boundaries.
\begin{lemma}\label{support}
Every ergodic non-atomic shift invariant measure on $\overline{\Sigma}$ is supported by the set  $\Sigma$.
\end{lemma}
\begin{proof}

We suppose that  $\mu$ is an ergodic shift-invariant measure such that $\mu(\partial\Sigma)>0$.
Let $\Omega_i := \{\x\in\overline{\Omega}:x_{i}  \text{ is a vertex}\}$  and $\Sigma_i = \{\a: X(\a) \subset \Omega_i\}$.
Clearly
 $\partial \Sigma = \cup_i  \Sigma_i$, thus our assumption implies that $\mu(\partial\Sigma_{i})>0$ for some $i$.
Now by ergodicity, $\mu$-almost every point $\a \in \overline{\Sigma}$ should visit this $\Sigma_{i}$ an infinite number of times, i.e.\  for each $\x \in X(\a)$ 
there are times $j < k$ so that $\x_j$ and $\x_k$ are vertices, i.e the orbit of $\x$ is 
generalized diagonal. But the set of generalized diagonals is countable, by Proposition \ref{extend}
their lift is countable, and thus  the measure is atomic.
\end{proof}

Next we have
\begin{lemma}\label{zero'}
For every ergodic shift invariant measure $\mu$ supported by the set $\Sigma$ 
the entropy $h_{\mu}(\Sigma)$ 
is equal to zero.
\end{lemma}

\begin{proof}
We consider the time zero partition $\xi$ of $\Sigma$ and the partition $\xi^- : = \vee_{j=0}^{\infty} \sigma^{i}\xi$.
An element $P \in \xi^-$ corresponds to a code in $\Sigma^-$, there are two possibilities, either this
code is periodic, or not.

(i) The code is periodic, then Proposition \ref{unique} tells us that $X(P)$ consists of a finite union
of strips of periodic orbits, all with the same past code; thus all with the same future code. In particular
$P$ is a single point.

(ii) If the code is not periodic, then  Proposition \ref{unique} tells us the $X(P)$ is a single point in $\Omega$;
thus $P$ is a single point.

Now we apply Rokhlin's theory of entropy (\cite{R}),  we have shown that  $\xi$ is a one sided generating partition ($\xi^-$ is the partition into points), thus
$$h_{\mu}(\Sigma)=h_{\mu}(\Sigma,\xi) = H(\sigma \xi /\xi^{-})=0.$$
 \end{proof}

{\noindent \bf 2.5 \ Proof of entropy theorems.\\}

\begin{proofof}{Theorem \ref{zero}}
 $h_{top}(\overline{\Sigma},\sigma) = 0$ follows immediately from Lemma \ref{support}, Lemma \ref{zero'} and the variational principle. We have
$h_{top}(\overline{\Omega},\sigma) \le h_{top}(\overline{\Sigma},\sigma) $ since $(\overline{\Omega},\sigma)$ is a continuous factor  of $(\overline{\Sigma},\sigma)$
(Proposition \ref{extend}).
Similarly, since $(\A,T)$ is a continuous factor  of $({\Omega},\sigma)$ (Proposition \ref{attractor}), we have
$h_{top}(\A,T) \le h_{top}(\Omega,\sigma) \le h_{top}(\overline{\Omega},\sigma) $.
\end{proofof}

Let $f:X \to X$ be a continuous map of a compact topological space $X$.
Let $(\overline{X},f)$ denote the inverse limit  of the map $f$.
The {\it non wandering set} $NW(X)$ of a continuous
map $f:X \to X$ of a compact topological space $X$ is the set $\{x \in X:$ for each neighborhood 
$\mathcal{U}$ of $x$,  $\exists n \ne 0$ such that $f^n \mathcal{U} \cap \mathcal{U} \ne \emptyset\}$. 

\begin{proposition}\label{NW}
If $h_{top}(\overline{X},f)= 0$ the $h_{top}(X,f) = 0$.
\end{proposition}

 \begin{proofof}{Theorem \ref{zero''}}
 Theorem \ref{zero''} follows immediately by combining Theorem \ref{zero} and the proposition.\end{proofof}
 
 \begin{proofof}{Proposition \ref{NW}}
 We use a theorem of Bowen (\cite{Bo},Theorem 2.4) $h_{top}(X,f) = h_{top}(NW(X),f)$.
 Let $p : \overline{X} \to X$ be the natural projection map. 
Clearly $p \circ f (\x) = f \circ p (\x)$ for any $\x \in \overline{X}$,
which implies $h_{top}(NW(\overline{X}),f) \ge h_{top}(p(NW(\overline{X})),f)$;
therefore  $h_{top}(p(NW(\overline{\Sigma})),f)=0$.
But $p(NW(\overline{X})) =NW(X)$ (\cite{AH} Theorem 3.5.1(5)).
Thus applying once again the above mention theorem of Bowen  we conclude 
that  
 $0 =  h_{top}(NW(X),f) = h_{top}(X,f)$.
\end{proofof}

\section{The proofs of the complexity results}

Fix a $q$-gon $Q$ with $r$ spy mirrors. 
A {\em spy generalized diagonal} is an orbit which starts at a spy mirror (thus it has two backwards continuations)
and ends in a vertex of $Q$, thus its orbit has two backwards continuations with different codes.  Consider a spy generalized diagonal,
since the length of the spy generalized diagonal is finite, slightly varying the angle of arrival 
at the vertex will yield a spy generalized diagonal with the same code (see Figure \ref{Un}).

\begin{figure}[h]
\vspace{-3cm}
\includegraphics[width=7.1cm,height=10cm]{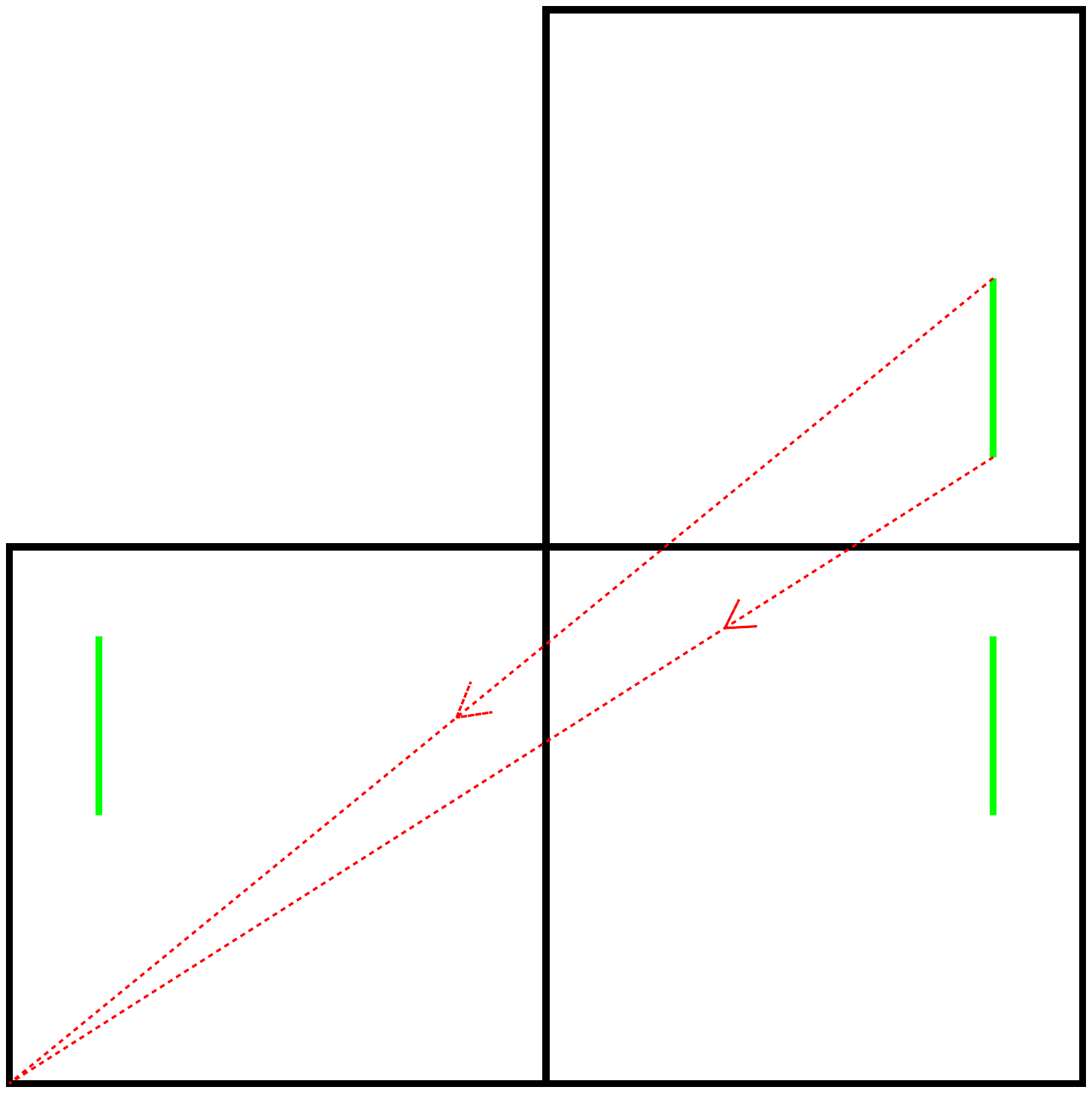}
\put(-100,145){$b$}
\put(-70,185){$a$}
\put(-32,145){$d$}
\put(-70,113){$c$}
\put(-100,75){$b$}
\put(-32,75){$d$}
\put(-180,75){$d$}
\put(-138,113){$c$}
\put(-138,45){$a$}
\put(-70,45){$a$}
\vspace{-1.5cm}
\caption{All the orbit segments at the wedge are spy generalized diagonals with the same code $cb$. This word
has 4 extension $e^{\pm}cba$ and $e^{\pm}cbd$ where $e^+$ and $e^-$ refer to the two sides of the one-sided mirror.}
\label{Un}
\end{figure} 

Let $N_{spy}(n)$ denote the number of such families of length at most $n$.
The set of spy
generalized diagonals with a given code form a ``sector'' at the vertex, the boundaries of this sector
are generalized diagonal of length at most $n$.

 The boundary of a spy generalized diagonal sector of length $n$ consists of 
two generalized diagonals of length at most $n$, one bounding from the left and one from the right (here left and right are with respect to a fixed orientation, say clockwise, at the vertices of arrival of
the generalized diagonals).  Thus we can define an injective map, the ``right bounding generalized diagonal'', from families of spy generalized diagonals of length $n$ to generalized diagonal of length at most $n$ yielding

\begin{proposition} \label{prop}
$N_{spy}(n) \le  \sum_{i=0}^{n} N_{vert}(i)$.
\end{proposition}

\begin{proofof}{Theorem \ref{thm1}}
The main technical tool will be a variant of Cassaigne's formula \cite{C} developed in \cite{ST}.
Let $\A$ be a finite alphabet, $\L \subset \A^{\mathbb{N}}$ be a language, $\L(n)$ the set of words of length $n$ which appear in $\L$, and $p(n) := \# \L(n)$.
Note that $p(0) = \# \{\emptyset\} = 1$.
For any $n \ge 0$ let $s(n) := p(n+1) - p(n),$ and thus $$p(n) = 1 + \sum_{j=0}^{n-1} s(j).$$
For $u \in \L(n)$ let
\begin{eqnarray*}
m_l(u) & := & \#\{a \in \A: au \in \L(n+1)\},\\
m_r(u) & := & \#\{b \in \A: ub \in \L(n+1)\},\\
m_b(u) & := & \#\{(a,b) \in \A^2: aub \in \L(n+2)\}.
\end{eqnarray*}
We remark that while $m_r(u) \ge 1$ the other two quantities can be $0$.
A word $u \in \L(n)$ is called {\em left special} if $m_l(u) > 1$, {\em right special}
if $m_r(u) > 1$ and {\em bispecial} if it is left and right special.
Let $\BL(n) := \{u \in \L(n): u \text{ is bispecial}\}.$
Let $\L_{np}(n):= \{v \in \L(n): m_l(v) = 0 \}$.

In \cite{ST} we showed that 
$$s(j+1) - s(j) = 
\sum_{v \in \BL(j)} \Big (m_b(v) - m_l(v) - m_r(v) + 1 \Big )   - \sum_{v \in \L_{np}(j):m_r(v) >1}  
\Big (m_r(v) - 1 \Big ). $$

Now we turn to the interpretation of this formula in the case that  $v$ is a bispecial 
word appearing in a $q$-gon $Q$ with $r$ spy mirrors.   
We make a worst case estimate, apriori the collection of orbits with the code $v$ can hit all sides and all spy mirrors (forwards or backwards)

\begin{eqnarray*}
2 & \le   m_l(v) \le & (q+2r)\\
2  & \le     m_r(v)  \le &  (q+2r)\\
2 & \le   m_b(v)   \le  & (q+2r)^2
\end{eqnarray*}
Thus
$$s(j+1) - s(j) \le ((q+2r)^2-3)\#\BL(j)$$
We have  $p(0) = 1$ (the empty set) and $p(1) = q+2r$, thus $s(0) = p(1) -1 = q+2r-1$ and the method of telescoping sums yields
$$
s(n)   \le (q+2r - 1)  + ((q+2r)^2-3)\sum_{j=0}^{n-1} \#\BL(j).$$
To each bispecial word $v$ corresponds a collection of generalize diagonals and/or spy generalized diagonal sectors of the same combinatorial length and code. Since the collections are determined by their code they are disjoint, thus 
$$\sum_{j=0}^{n-1} \#\BL(j) \le N_{vert}(n) + N_{spy}(n) \le N_{vert}(n) + \sum_{i=0}^n N_{vert}(i) \le 2 \sum_{i=0}^n N_{vert}(i), $$
and thus
$$s(n) \le (q+ 2r - 1)  + 2((q+2r)^2-3) \sum_{i=0}^n N_{vert}(i).$$

Again by the method of telescoping sums we have
\begin{eqnarray*} p(n) \le  1 + \sum_{j=0}^{n-1} \left ((q+2r - 1)  + 2((q+2r)^2-3)\sum_{i=0}^{j}N(i)) \right ) \\
=  1 + (q+2r- 1)n  +  \left (2((q+2r)^2-3)\sum_{j=0}^{n-1} \sum_{i=0}^{j} N_{vert}(i) \right ).
\end{eqnarray*}
\end{proofof}

\begin{proofof}{Theorem \ref{thm2}}
Consider a rational polygon $P$ without spy mirrors. H. Masur in \cite{M} has shown that the number $N_g(t)$ of generalized
diagonals of geometric length $t$ satisfies $N_g(t) \le Ct^2$ for some constant $C = C(P)$ for all $t  \ge 0$.  
By elementary
reasoning there is a constant $B > 1$ such that $ B^{-1} \le N_{vert}(n)/N_g(n) \le  B$, thus there
is a constant $K = K(P)$ so that 
$N_{vert}(n) \le K n^2$ for all $n \ge 0$.

Suppose that the polygon $Q$ with spy mirrors is a $k$ fold cover of the rational polygon $P$.
By the symmetry in the definition of $Q$ each generalized diagonal in $Q$  projects to a generalized
diagonal in $P$, thus $N_{vert}^Q(n) \le k N^Q_{vert}(n) \le kK n^2$. Thus the Theorem follows from  Theorem \ref{thm1}.
\end{proofof}

\begin{proofof}{Theorem \ref{thm3}}
D. Scheglov in \cite{Sch} has shown that for a typical triangle, for every $\varepsilon > 0$ there exists a constant $C > 0$ so that 
$N_{vert}(n) \le Ce^{n^{\varepsilon}}$  for all $n  \ge 0$.  

Suppose that the polygon $Q$ with spy mirrors is a $k$ fold cover of  typical triangle $P$.
As in the proof of Theorem \ref{thm2}, each generalized diagonal in $Q$  projects to a generalized
diagonal in $P$, thus $N_{vert}^Q(n) \le k N^Q_{vert}(n) \le kC  e^{n^\varepsilon}$ and we again use Theorem
\ref{thm1} to conclude.
\end{proofof}

\begin{proofof}{Theorem \ref{thm4}}
First suppose the square is $1 \times 1$ and  with a single   spy mirror having $x$ coordinate $a$.
Let $X$ denote the torus ($2 \times 2$) obtained by unfolding the square with the one-sided mirrors identified as in the figure \ref{Sq}.
 
\begin{figure}[h]
\vspace{-3cm}
\includegraphics[width=7.1cm,height=10cm]{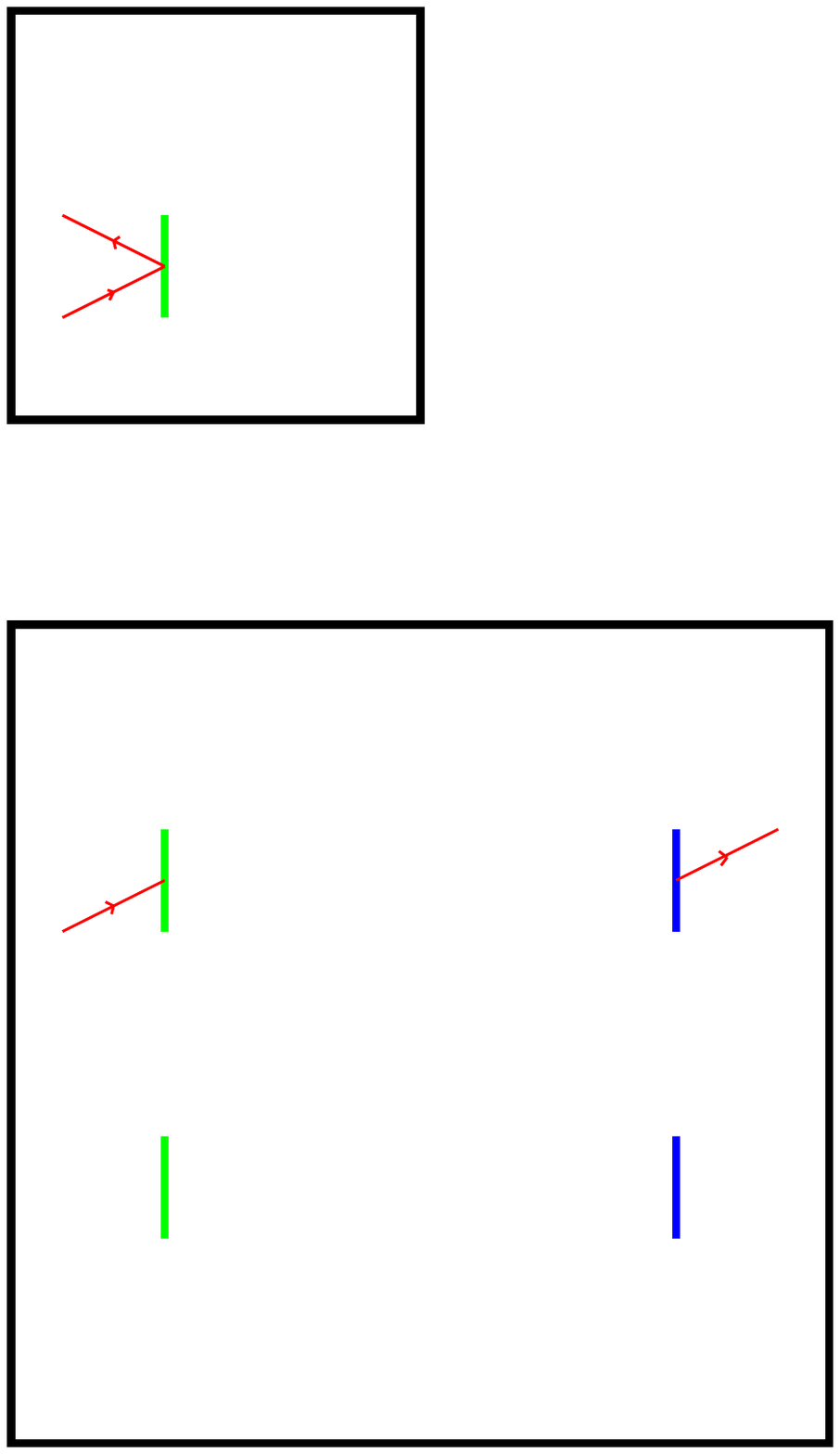}
\vspace{-1.5cm}
\caption{A square billiard with the vertical mirror: table and unfolding}
\label{Sq}
\end{figure} 

Now we construct an explicit universal cover of $X$, first
we unfold the square to a torus and then take the universal cover of the torus by the plane. Next we identify
the one-sided mirrors at $x=2-a +2n$ with the one-sided mirror at $x =  2 + a + 2n$, i.e.\  all jumps are to the right
and of the form $2a$ 
(see Figure \ref{Cover}).

\begin{figure}[h]
\vspace{-3.8cm}
\includegraphics[width=8cm,height=11cm]{Cover.pdf}
\put(-129,195){\vector(0,-1){35}}
\put(-160,200){\footnotesize{$x=2-a+2n$}}
\put(-103,205){\vector(0,-1){45}}
\put(-120,210){\footnotesize{$x=2 + a+2n$}}
\vspace{-1cm}
\caption{A square billiard with the vertical mirror: universal cover}
\label{Cover}
\end{figure}

Now consider the generalized diagonals starting at the origin which arrives at the point $(m,n)$ ($m>0$ and $n>0$) which do not pass through
any other vertex in-between. Label the squares traversed by the generalized diagonals by the coordinates of their bottom left corners.
By the explicit construction of our universal cover, when the generalized diagonal is in the square with bottom left
corner $(i,j)$ then either we cross a horizontal side and are in square $(i,j+1)$ or we cross a spy mirror, or a vertical side and are in square $(i+1,j)$.
Thus we cross $n$ horizontal sides $y=1,y=2,\dots,y=n-1$, $f$ vertical sides and $g$ vertical mirrors with $f+g = m-1$. We conclude that 
the combinatorial length of any such generalized diagonal is $(m-1)+f + g = m + n-2$. We want to estimate how many such generalized diagonals we can have. Clearly there is at most one which hits no spy mirrors, it
is the line segment of slope $n/m$ starting at the origin and ending at $(m,n)$, and it is a generalized diagonal if and only if this segment does not
reflect form any spy mirrors. 

Now we claim that, for each $0 \le j \le m-1$
 there is at most one generalized diagonal connecting $(0,0)$ to $(m,n)$
which hits exactly $j$ spy mirrors.  In fact, the slope of such a generalized diagonal must be $n/(m - j2a)$, and then if we start an orbit
segment at the origin with this slope it is a generalized diagonal (arrives at $(m,n)$)  if and only if it hits exactly $j$ spy mirrors. 

Thus there are at most $m$ generalized diagonals connecting $(0,0)$ to $(m,n)$, all other cases of pairs of vertices are similar.
There are $const\cdot N^2$ lattice points and ends of spy mirrors at a distance $N$ of the origin, 
 yielding
a cubic estimate for the number of generalized diagonals and thus  a quintic estimate on the complexity by Theorem \ref{thm1}.

Now consider the general case.  Let $x=a_i$, $1 \le i \le k$ be the $x$ coordinates of the finite collection of vertical mirrors (there can be several
mirrors with the same $x$ coordinate).
The reflecting sides of the mirrors are not necessarily the same.
The universal cover identifies the copy of each mirror in $x = 2 - a_i  + 2n$ with the associated one-sided mirror at $x = 2 + a_i + 2n$ for
mirrors with the reflecting side of the left say, and if the right side is  reflecting then $x = a_i + 2n$ with the associated mirror at   $x = 2 - a_i + 2n$.  As we trace a trajectory it always
jumps to the right, by $b_i := 2a_i$ in the first case, and $b_i := 2-2a_i$ in the second case.

Consider the generalized diagonal connecting $(0,0)$ to $(m,n)$.  Let $0 \le j_i \le m-1$ denote the total number spy mirrors with $x$-coordinate 
$a_i$ which the generalize diagonal hits, and $0 \le j_0 \le m-1$
denote the number of  vertical sides it crosses.
For each $0 \le \ell < m-2$ it hits exactly one spy mirror or one vertical mirror in each rectangle $\ell < x \le \ell +1$, thus  
\begin{equation}
\sum_{i=0}^k j_i = m-1\label{s}
\end{equation}
 (note that it can not hit a vertical mirror in the
rectangle $m-1 < x \le m$ and arrive at the point $(m,n)$ and by definition it arrives, but does not cross the side $x=m$). 
Furthermore  the total horizontal distance
travelled by the generalized diagonal is $n - \sum_{i=1}^k j_i b_i$. Since the vertical distance it travelled is $m$ it has slope $n/(m - \sum_i j_i b_i)$.

We claim by induction that the number $Q_k(m)$ of integer solutions  of \eqref{s} with each $0 \le j_i \le m-1$
satisfies $Q_k(m) \le m^k/k!$ for all $m \ge 1$. For $k =1$ equality holds, this was used  above.
Suppose that this is true for some fixed $k$. To estimate $Q_{k+1}(m)$ suppose that $k_0 = p
\in \{0,1,\cdots,m-1\}$,
then $\sum_{i=1}^k k_i = m - 1 -p$ which has $Q_k(m-p) \le (m-p)^k/k!$ solutions. Thus 
$Q_{k+1}(m) \le \sum_{p=0}^{m-1} (m-p)^k/k! \le m^{k+1}/(k+1)!$.

As above we have shown that for each choice of the set $\{j_i\}$ we have at most one generalized diagonal connecting the origin to $(m,n)$.
Thus there are at most $O(m^k)$  generalized diagonals connecting $(0,0)$ to $(m,n)$, all other cases of pairs of vertices are similar. Again 
there are at most  $O(N^2)$ lattice points and ends of spy mirrors at a distance of at most $N$ from the origin, yielding
at most $O(N^{k+2})$  generalized diagonals of length $N$ and thus   the complexity is at most $O(N^{k+4})$ by Theorem \ref{thm1}. 
\end{proofof}

Remark, we actually prove a slightly stronger theorem
\begin{theorem}
Suppose that $Q$ is  the square with a finite number of  vertical spy mirrors cottoned in $k$ vertical lines.
Then the total complexity satisfies
$p(n) \le Cn^{k+4}$ for all $n$. 
\end{theorem}

\section{Acknowledgements.}  We gratefully acknowledge the support of ANR Perturbations. The first author was also partially supported by Dynasty Foundation.

\end{document}